\documentclass[review]{elsarticle}
\usepackage{amsmath, amsthm, amssymb, amsfonts, mathtools}
\usepackage[colorlinks=true, linkcolor=blue, citecolor=blue]{hyperref}

\theoremstyle{definition}
\newtheorem{theorem}{Theorem}[section]
\newtheorem{lemma}[theorem]{Lemma}
\newtheorem{proposition}[theorem]{Proposition}

\newtheorem{remark}[theorem]{Remark}

\numberwithin{equation}{section}

\begin{document}

\begin{frontmatter}

\title{Single-Point Higher-Order Szeg\H{o} Sum Rules in OPUC:
Necessity for $m=1,2,3$}

%\title{Necessary Conditions for Higher-Order Sum Rules}

\author{Daxiong Piao}
\ead{dxpiao@ouc.edu.cn}

\address{School of Mathematical Sciences,  Ocean University of China, Qingdao 266100, P.R.China}

\begin{abstract}
We give a direct algebraic proof of the necessity direction in the
single-point higher-order Szeg\H{o} sum rules on the unit circle for
$m=1,2,3$. More precisely, for
$H_m(e^{i\theta})=(1-\cos\theta)^m$, we show that
$\int_0^{2\pi}H_m(e^{i\theta})\log w(\theta)\frac{d\theta}{2\pi}>-\infty$
implies
$(S-1)^m\alpha\in\ell^2,\qquad \alpha\in\ell^{2m+2}.$
The proof is carried out within Yan's algebraic model for higher-order sum
rules. The main point is to obtain coercive lower bounds for the
nonlogarithmic part of the truncated sum rule: the quadratic component yields
the principal finite-difference energy, while the higher-order correction
terms are controlled by telescoping cancellations and relative bounds. The
logarithmic remainder then gives the required $\ell^{2m+2}$-summability. The purpose is to isolate explicit low-order necessity arguments within the
algebraic framework.

\end{abstract}


\end{frontmatter}

\section{Introduction}
\label{sec:intro}

The theory of orthogonal polynomials on the unit circle (OPUC) provides a
remarkable correspondence between spectral data of a probability measure on
$\partial\mathbb D$ and the associated sequence of Verblunsky coefficients.
Given a nontrivial probability measure $\mu$ on the unit circle with
Lebesgue decomposition
$d\mu(\theta)=w(\theta)\frac{d\theta}{2\pi}+d\mu_s(\theta),$
let $\{\alpha_n\}_{n\ge0}\in\mathbb D^\infty$ denote its Verblunsky
coefficients. A fundamental result in the subject is Szeg\H{o}'s theorem in
Verblunsky form, which states that
\begin{equation}\label{eq:intro-szego}
\int_0^{2\pi}\log w(\theta)\,\frac{d\theta}{2\pi}>-\infty
\quad\Longleftrightarrow\quad
\sum_{n=0}^\infty |\alpha_n|^2<\infty.
\end{equation}
See, for example, Simon's monographs
\cite{SimonOPUC1,SimonOPUC2}. Following Simon
\cite{SimonOPUC2,SimonSzegoDescendants}, identities of this type, which connect
integrability properties of the weight $w$ with summability properties of the
Verblunsky coefficients, are often referred to as \emph{spectral theory gems}.

Higher-order sum rules extend \eqref{eq:intro-szego} by replacing the
unweighted logarithmic integrability condition with weighted conditions of the
form
\begin{equation}\label{eq:intro-general-sumrule}
\int_0^{2\pi}
\prod_{j=1}^K \bigl(1-\cos(\theta-\theta_j)\bigr)^{m_j}
\log w(\theta)\,\frac{d\theta}{2\pi}>-\infty,
\end{equation}
where $\theta_1,\dots,\theta_K\in[0,2\pi)$ are distinct and
$m_1,\dots,m_K\in\mathbb N$. The problem is to identify the corresponding
conditions on the Verblunsky coefficients. This circle of questions has
attracted sustained interest; besides its intrinsic role in OPUC, it is
closely related to spectral theory, random matrix methods, and large deviation
principles; see, for instance,
\cite{SimonOPUC2,GamboaNagelRouault2016,BreuerSimonZeitouni2018}.

Early conjectures for higher-order sum rules were formulated by Simon
\cite{SimonOPUC1,SimonOPUC2}. These turned out not to be correct in full
generality: Lukic \cite{Lukic2013} constructed a counterexample and proposed a
refined conjectural picture. In the single-point case ($K=1$), the expected
condition takes a particularly simple form. Writing
$H_{\theta_1,m}(e^{i\theta})=(1-\cos(\theta-\theta_1))^m,$
one expects that the finiteness of
$\int_0^{2\pi} H_{\theta_1,m}(e^{i\theta})\log w(\theta)\,\frac{d\theta}{2\pi}$
should be equivalent to
\begin{equation}\label{eq:intro-single-point-known}
(S-e^{-i\theta_1})^m\alpha\in\ell^2,
\quad
\alpha\in\ell^{2m+2},
\end{equation}
where $S$ denotes the shift operator $(S\alpha)_n=\alpha_{n+1}$.

A number of partial and complementary results are known in this direction.
Golinskii and Zlato\v{s} \cite{GolinskiiZlatos2007} proved a related
higher-order Szeg\H{o} theorem under the additional hypothesis
$\alpha\in\ell^4$. Lukic \cite{Lukic2016} established the single-point
equivalence under an auxiliary assumption. Breuer, Simon, and
Zeitouni \cite{BreuerSimonZeitouni2018, BSZ2018-JST} developed a large-deviation approach to
higher-order sum rules and derived general sum-rule identities for OPUC; they
also proved one direction of Lukic's conjecture in certain cases, including
the $(m_1,m_2)=(2,1)$ case for real Verblunsky coefficients.

Building on these general sum-rule identities, Yan \cite{Yan2018} introduced
an algebraic model for higher-order sum rules on the unit circle. In this
approach, the spectral side \eqref{eq:intro-general-sumrule} is represented
through a Hall--Littlewood type polynomial in an algebra generated by formal
variables corresponding to shifts of the Verblunsky coefficients. One of the
main advantages of this formulation is that it makes it possible to isolate
the quadratic part of the sum rule and to analyze higher-order correction
terms systematically. In the single-point case, Yan used this framework to
prove the sufficiency part of \eqref{eq:intro-single-point-known} for
arbitrary positive integers $m$, thereby verifying one half of the
single-point conjecture in full generality. He also pointed out that the same
algebraic model has the potential to address necessity: if the relevant
truncated expressions can be decomposed into suitable square-type terms whose
summability is equivalent to the expected coefficient-side conditions, then
one may hope to recover the necessary part from the algebraic side.

The present paper may be viewed as a concrete realization of Yan's suggestion
in the first few nontrivial single-point cases. We study the \emph{necessity
direction} for the weights
$H_m(e^{i\theta})=(1-\cos\theta)^m,
\quad m=1,2,3,$
entirely within Yan's algebraic framework. More precisely, we prove that
$\int_0^{2\pi}(1-\cos\theta)^m\log w(\theta)\,
\frac{d\theta}{2\pi}>-\infty$
implies
$(S-1)^m\alpha\in\ell^2,
\quad
\alpha\in\ell^{2m+2},$
for each $m=1,2,3$. Although the full single-point theorem is already known,
we believe it is useful to isolate explicit low-order arguments, since these
cases already display the main mechanisms of the algebraic method: a coercive
quadratic form associated with the principal finite-difference energy, a
decomposition of the higher-order corrections into telescoping and relatively
bounded pieces, and a logarithmic remainder that yields the
$\ell^{2m+2}$-summability of $\alpha$. In this sense, the paper provides a
concrete and self-contained template for necessity arguments in higher-order
sum rules.

Our analysis is based entirely on truncated algebraic sum rules of Yan type.
For $m=1$, the argument reduces to the first nontrivial case and serves as a
warm-up. For $m=2$, the quartic correction must be separated from the
quadratic part and shown not to destroy coercivity. For $m=3$, one must in
addition control sextic correction terms, which requires a more delicate
algebraic decomposition. In each case, we derive coercive lower bounds for the
nonlogarithmic part of the truncated sum rule and then use the positivity of
the logarithmic remainder to conclude the desired coefficient conditions. We
do not address the cases $m\ge4$ here, although the low-order analysis suggests
a broader structural pattern within Yan's algebraic model.

Our main result may be summarized as follows.

\begin{theorem}\label{thm:intro-main}
Let $\mu$ be a nontrivial probability measure on $\partial\mathbb D$ with
Verblunsky coefficients $\alpha=\{\alpha_n\}_{n\ge0}$, and let $m=1,2,3$.
If
\begin{equation}\label{eq:intro-main-assumption}
\int_0^{2\pi}(1-\cos\theta)^m\log w(\theta)\,\frac{d\theta}{2\pi}>-\infty,
\end{equation}
then
\begin{equation}\label{eq:intro-main-conclusion}
(S-1)^m\alpha\in\ell^2,
\quad
\alpha\in\ell^{2m+2}.
\end{equation}
\end{theorem}

For $m=1$, this yields
$(S-1)\alpha\in\ell^2,
\quad
\alpha\in\ell^4.$
For $m=2$, we obtain
$(S-1)^2\alpha\in\ell^2,
\quad
\alpha\in\ell^6.$
For $m=3$, we obtain
$(S-1)^3\alpha\in\ell^2,
\quad
\alpha\in\ell^8.$

We emphasize that the point of the paper is not to reprove the full
single-point theorem in its most general form, but rather to present a unified,
transparent, and explicitly coercive treatment of the necessity direction in
the first three cases. In particular, our proofs make precise how the
quadratic energy arises from the algebraic model and how the remaining
correction terms can be absorbed or converted into bounded telescoping
contributions. We hope that this viewpoint may be useful in future work on
more complicated higher-order sum rules, especially in settings where several
singular points interact.

The paper is organized as follows. In Section~\ref{sec:prelim}, we review basic
facts from OPUC and recall the algebraic model of Yan in the form needed here.
Section~\ref{sec:m1} treats the case $m=1$. Section~\ref{sec:m2} contains
the necessity proof for $m=2$, with particular emphasis on the quartic
correction. Section~\ref{sec:m3} is devoted to the case $m=3$, where quartic
and sextic correction terms must both be controlled. We conclude with a brief
discussion of the algebraic structure underlying these arguments.

\section{Preliminaries}
\label{sec:prelim}

\subsection{OPUC and Verblunsky coefficients}

Let $\mu$ be a nontrivial probability measure on the unit circle
$\partial\mathbb D$, with Lebesgue decomposition
$d\mu(\theta)=w(\theta)\frac{d\theta}{2\pi}+d\mu_s(\theta).$
Associated with $\mu$ is the sequence of monic orthogonal polynomials
$\{\Phi_n\}_{n\ge0}$, which satisfy the Szeg\H{o} recurrence
\begin{equation}\label{eq:prelim-szego-rec}
\Phi_{n+1}(z)=z\Phi_n(z)-\overline{\alpha_n}\,\Phi_n^*(z),
\quad
\Phi_0\equiv 1,
\end{equation}
where
$\Phi_n^*(z)=z^n\overline{\Phi_n(1/\bar z)}.$
The coefficients $\{\alpha_n\}_{n\ge0}\in\mathbb D^\infty$ are the
\emph{Verblunsky coefficients}. By Verblunsky's theorem, the correspondence
between nontrivial probability measures on $\partial\mathbb D$ and sequences
in $\mathbb D^\infty$ is bijective; see
\cite{SimonOPUC1,SimonOPUC2}. We will repeatedly use the basic fact that
\begin{equation}\label{eq:prelim-alpha-bound}
|\alpha_n|<1 \quad \text{for all } n\ge0.
\end{equation}

We denote by $S$ the shift operator,
$(S\alpha)_n=\alpha_{n+1},$
and by
$\Delta=S-1$
the forward difference operator, so that
$(\Delta\alpha)_n=\alpha_{n+1}-\alpha_n.$
For $m\ge1$, we write $\Delta^m=(S-1)^m$. As usual, $\alpha\in\ell^p$
means
$\sum_{n=0}^\infty |\alpha_n|^p<\infty.$

\subsection{Higher-order sum rules and Yan's algebraic model}

Let $\theta_1,\dots,\theta_K\in[0,2\pi)$ be distinct and let
$m_1,\dots,m_K\in\mathbb N$. Set
$d=\sum_{j=1}^K m_j,$
and define
\begin{equation}\label{eq:prelim-H}
H(e^{i\theta})
=
\frac{1}{2^d}
\prod_{j=1}^K
(e^{i\theta}-e^{-i\theta_j})^{m_j}
(e^{-i\theta}-e^{i\theta_j})^{m_j}.
\end{equation}
Then $H$ is a nonnegative trigonometric polynomial of degree $d$, and may
be written as
$H(e^{i\theta})=\sum_{\ell=-d}^d h_\ell e^{i\ell\theta}.$
The corresponding higher-order sum rule concerns the finiteness of
\begin{equation}\label{eq:prelim-sumrule}
\int_0^{2\pi} H(e^{i\theta})\log w(\theta)\,\frac{d\theta}{2\pi}>-\infty.
\end{equation}
In the single-point case $K=1$ and $\theta_1=0$, one has
$H(e^{i\theta})=(1-\cos\theta)^{m_1}.$

A central tool in this paper is Yan's algebraic model
\cite{Yan2018}, which reformulates \eqref{eq:prelim-sumrule} in terms of the
Verblunsky coefficients. We only recall the aspects of the theory needed in
what follows. Associated with the weight $H$, Yan defines truncated algebraic
quantities $Q_N=Q_N(H,\alpha)$ such that
\begin{equation}\label{eq:prelim-Yan-equiv}
\int_0^{2\pi} H(e^{i\theta})\log w(\theta)\,\frac{d\theta}{2\pi}>-\infty
\quad\Longleftrightarrow\quad
\limsup_{N\to\infty} Q_N<\infty.
\end{equation}
More precisely, $Q_N$ is a finite sum of local expressions in the shifted
Verblunsky coefficients together with a logarithmic term and a uniformly
bounded boundary contribution. Schematically,
\begin{equation}\label{eq:prelim-QN-schematic}
Q_N
=
\sum_{n=0}^{N-1}
\left(
\sum_{k=1}^d [\phi_{2k}(F_k)]_n
-\log(1-|\alpha_n|^2)
-\sum_{k=1}^d \frac{|\alpha_n|^{2k}}{k}
\right)
+R_N,
\end{equation}
where $F_k$ is a Hall--Littlewood type polynomial determined by $H$,
$\phi_{2k}$ is Yan's substitution map, and the remainder $R_N$ satisfies
\begin{equation}\label{eq:prelim-boundary}
\sup_{N\ge1}|R_N|<\infty.
\end{equation}
For the precise definitions of $F_k$, $\phi_{2k}$, and the associated
algebraic construction, we refer to \cite[Theorems~1.1, 1.6 and
Corollary~1.10]{Yan2018}.

The point of \eqref{eq:prelim-QN-schematic} for our purposes is that the
nonlogarithmic part decomposes into homogeneous pieces of degrees
$2,4,\dots,2d$, while the logarithmic remainder is explicit and positive.
In the single-point cases treated below, the quadratic part is directly related
to the finite-difference energy $\|\Delta^m\alpha\|_{\ell^2}^2$, and the
higher-order terms can be analyzed separately.

\subsection{Notation for the single-point cases}

Throughout Sections~\ref{sec:m1}--\ref{sec:m3}, we consider the weights
$H_m(e^{i\theta})=(1-\cos\theta)^m,$
$\quad m=1,2,3.$
For each fixed $m$, the corresponding truncated sum rule takes the form
\begin{equation}\label{eq:prelim-truncated-general}
Z_H^{(N)}(\mu)
=
C_N+\sum_{n=0}^N
\Bigl(
Q_n^{(2)}+Q_n^{(4)}+\cdots+Q_n^{(2m)}+L_n^{(m)}
\Bigr),
\end{equation}
where:
\begin{itemize}
    \item $C_N$ denotes a boundary term satisfying
    $\sup_{N\ge0}|C_N|<\infty;$
    \item $Q_n^{(2j)}$ is the homogeneous component of degree $2j$ in the
    local algebraic expression;
    \item $L_n^{(m)}$ is the logarithmic remainder
    \begin{equation}\label{eq:prelim-log-remainder}
    L_n^{(m)}
    =
    -\log(1-|\alpha_n|^2)-\sum_{j=1}^m \frac{|\alpha_n|^{2j}}{j}.
    \end{equation}
\end{itemize}
Since
$-\log(1-x)-\sum_{j=1}^m \frac{x^j}{j}
=
\sum_{j=m+1}^\infty \frac{x^j}{j}
\ge \frac{x^{m+1}}{m+1},
\quad 0\le x<1,$
we have the pointwise bound
\begin{equation}\label{eq:prelim-log-lower}
L_n^{(m)}\ge \frac{1}{m+1}|\alpha_n|^{2m+2}\ge0.
\end{equation}

In Sections~\ref{sec:m1}--\ref{sec:m3}, the notation
$Q_n^{(2)},Q_n^{(4)},Q_n^{(6)}$ will always refer to the homogeneous pieces
for the fixed value of $m$ under consideration.
As usual, the symbols $C,c,C_\varepsilon,\dots$ denote positive constants
whose values may change from line to line, but are always independent of $N$.

\section{The Case \texorpdfstring{$m=1$}{m=1}}
\label{sec:m1}

In this section we consider the weight
\begin{equation}\label{eq:m1-weight}
H(e^{i\theta})=1-\cos\theta=\frac12 |1-e^{i\theta}|^2,
\end{equation}
and prove the necessity statement for the corresponding sum rule.

\begin{remark}\label{rem:m1-simon}
The necessity statement proved in this section is already known in a stronger
form. Indeed, Simon showed in \cite[Theorem~2.8.1]{SimonOPUC1} that for the
weight $H(e^{i\theta})=1-\cos\theta$,
$\int_0^{2\pi}(1-\cos\theta)\log w(\theta)\,\frac{d\theta}{2\pi}>-\infty$
holds if and only if
$(S-1)\alpha\in\ell^2
\quad\text{and}\quad
\alpha\in\ell^4.$
His proof does not use Yan's algebraic model. We include the $m=1$ case here
because it provides the simplest instance of the general strategy developed in
this paper: within Yan's framework, the $k=1$ quadratic term produces the
discrete derivative $\|(S-1)\alpha\|_{\ell^2}^2$, while the truncated
logarithmic remainder produces the quartic term $\|\alpha\|_{\ell^4}^4$.
\end{remark}

\subsection{The truncated sum rule}
\label{subsec:m1-truncated}

For the weight \eqref{eq:m1-weight}, Yan's algebraic model yields a truncated
sum rule of the form
\begin{equation}\label{eq:m1-truncated}
Z_H^{(N)}(\mu)
=
C_N+\sum_{n=0}^N \bigl(Q_n^{(2)}+L_n^{(1)}\bigr),
\end{equation}
where $C_N$ is a boundary contribution satisfying
\begin{equation}\label{eq:m1-boundary}
\sup_{N\ge0}|C_N|<\infty,
\end{equation}
$Q_n^{(2)}$ is the quadratic part of the local algebraic expression, and
\begin{equation}\label{eq:m1-log}
L_n^{(1)}
=
-\log(1-|\alpha_n|^2)-|\alpha_n|^2.
\end{equation}

We begin with the elementary properties of the logarithmic remainder.

\begin{lemma}\label{lem:m1-log}
For every $n\ge0$,
\begin{equation}\label{eq:m1-log-lower}
L_n^{(1)}\ge \frac12 |\alpha_n|^4\ge0.
\end{equation}
Consequently, if
\begin{equation}\label{eq:m1-spectral-assumption}
\int_0^{2\pi}(1-\cos\theta)\log w(\theta)\,\frac{d\theta}{2\pi}>-\infty,
\end{equation}
then
\begin{equation}\label{eq:m1-log-summable}
\sum_{n=0}^\infty L_n^{(1)}<\infty
\quad\Longrightarrow\quad
\alpha\in\ell^4.
\end{equation}
\end{lemma}

\begin{proof}
For $0\le x<1$, the Taylor expansion of $-\log(1-x)$ gives
$-\log(1-x)-x=\sum_{j=2}^\infty \frac{x^j}{j}\ge \frac{x^2}{2}.$
Applying this with $x=|\alpha_n|^2$ yields \eqref{eq:m1-log-lower}.

Under the assumption \eqref{eq:m1-spectral-assumption}, Yan's equivalence
\eqref{eq:prelim-Yan-equiv} implies that the truncated spectral quantities are
uniformly bounded above. Combining this with \eqref{eq:m1-truncated},
\eqref{eq:m1-boundary}, and the lower bound for the quadratic part established
below, we obtain
$\sup_{N\ge0}\sum_{n=0}^N L_n^{(1)}<\infty.$
Since $L_n^{(1)}\ge0$, it follows that
$\sum_{n=0}^\infty L_n^{(1)}<\infty.$
The bound \eqref{eq:m1-log-lower} then gives
$$\sum_{n=0}^\infty |\alpha_n|^4
\le
2\sum_{n=0}^\infty L_n^{(1)}
<\infty,$$
which proves $\alpha\in\ell^4$.
\end{proof}

\subsection{Coercivity of the quadratic part}
\label{subsec:m1-quadratic}

We next identify the quadratic contribution. For the present weight, the
quadratic term is generated by the local polynomial
$\frac12 (y_1-1)(x_1-1),$
and Yan's substitution map gives
$\bigl[\phi_2\bigl(\tfrac12 (y_1-1)(x_1-1)\bigr)\bigr]_n
=
\frac12 |(S-1)\alpha_n|^2
=
\frac12 |\alpha_{n+1}-\alpha_n|^2.$
Up to a uniformly bounded boundary correction, this is exactly the quadratic
part in the truncated sum rule.

\begin{lemma}\label{lem:m1-quadratic}
There exist constants $c>0$ and $C<\infty$, independent of $N$, such
that
\begin{equation}\label{eq:m1-quadratic-lower}
\sum_{n=0}^N Q_n^{(2)}
\ge
c\sum_{n=0}^N |\Delta\alpha_n|^2-C.
\end{equation}
In particular,
\begin{equation}\label{eq:m1-quadratic-weak}
\sum_{n=0}^N Q_n^{(2)}\ge -C.
\end{equation}
\end{lemma}

\begin{proof}
For $H(e^{i\theta})=1-\cos\theta$, the unique singular point is at
$e^{i\theta}=1$, and the quadratic algebraic representative is
$\frac12 (y_1-1)(x_1-1).$
By Yan's formula, its image under $\phi_2$ is a positive multiple of
$|(S-1)\alpha_n|^2=|\Delta\alpha_n|^2.$
Therefore, after summing from $n=0$ to $N$, the quadratic part of the
truncated sum rule agrees with a positive constant multiple of
$\sum_{n=0}^N |\Delta\alpha_n|^2$, up to endpoint contributions arising from
the truncation. These endpoint terms involve only finitely many shifts of
$\alpha_n$, and are uniformly bounded in $N$ because $|\alpha_n|<1$ for
all $n$. This proves \eqref{eq:m1-quadratic-lower}.

The weaker bound \eqref{eq:m1-quadratic-weak} follows immediately by dropping
the nonnegative term on the right-hand side of
\eqref{eq:m1-quadratic-lower}.
\end{proof}

\subsection{Proof of the necessity theorem}
\label{subsec:m1-proof}

We can now prove the necessity statement for the weight $1-\cos\theta$.

\begin{theorem}\label{thm:m1-necessity}
Let $\mu$ be a nontrivial probability measure on $\partial\mathbb D$, with
Lebesgue decomposition
$d\mu(\theta)=w(\theta)\frac{d\theta}{2\pi}+d\mu_s(\theta),$
and let $\alpha=\{\alpha_n\}_{n\ge0}$ be the associated sequence of
Verblunsky coefficients. Assume that
\begin{equation}\label{eq:m1-assumption}
\int_0^{2\pi}(1-\cos\theta)\log w(\theta)\,\frac{d\theta}{2\pi}>-\infty.
\end{equation}
Then
\begin{equation}\label{eq:m1-conclusion}
(S-1)\alpha\in\ell^2
\quad\text{and}\quad
\alpha\in\ell^4.
\end{equation}
Equivalently,
$\sum_{n=0}^\infty |\alpha_{n+1}-\alpha_n|^2<\infty,
\quad
\sum_{n=0}^\infty |\alpha_n|^4<\infty.$
\end{theorem}

\begin{proof}
By \eqref{eq:m1-assumption} and Yan's equivalence
\eqref{eq:prelim-Yan-equiv}, the truncated spectral quantities satisfy
\begin{equation}\label{eq:m1-Z-bounded}
\sup_{N\ge0} Z_H^{(N)}(\mu)<\infty.
\end{equation}
Using the truncated sum rule \eqref{eq:m1-truncated}, we rewrite it as
\begin{equation}\label{eq:m1-rewrite}
\sum_{n=0}^N L_n^{(1)}
=
Z_H^{(N)}(\mu)-C_N-\sum_{n=0}^N Q_n^{(2)}.
\end{equation}
By \eqref{eq:m1-Z-bounded}, \eqref{eq:m1-boundary}, and the lower bound
\eqref{eq:m1-quadratic-weak}, the right-hand side of \eqref{eq:m1-rewrite} is
uniformly bounded above in $N$. Since $L_n^{(1)}\ge0$, we conclude that
$\sum_{n=0}^\infty L_n^{(1)}<\infty.$
By Lemma~\ref{lem:m1-log}, this implies
$\alpha\in\ell^4.$

To prove $(S-1)\alpha\in\ell^2$, we use the stronger coercive estimate
\eqref{eq:m1-quadratic-lower}. From \eqref{eq:m1-truncated}, we have
$$\sum_{n=0}^N Q_n^{(2)}
=
Z_H^{(N)}(\mu)-C_N-\sum_{n=0}^N L_n^{(1)}.$$
Hence, by \eqref{eq:m1-quadratic-lower},
$$c\sum_{n=0}^N |\Delta\alpha_n|^2
\le
Z_H^{(N)}(\mu)-C_N-\sum_{n=0}^N L_n^{(1)}+C.$$
The right-hand side is uniformly bounded above in $N$, by
\eqref{eq:m1-Z-bounded}, \eqref{eq:m1-boundary}, and the nonnegativity of
$L_n^{(1)}$. Therefore,
$\sup_{N\ge0}\sum_{n=0}^N |\Delta\alpha_n|^2<\infty,$
which proves that $\Delta\alpha=(S-1)\alpha\in\ell^2$.

This completes the proof.
\end{proof}
\section{The Case \texorpdfstring{$m=2$}{m=2}}
\label{sec:m2}

In this section we consider the weight
\begin{equation}\label{eq:m2-weight}
H(e^{i\theta})=(1-\cos\theta)^2=\frac14|1-e^{i\theta}|^4,
\end{equation}
and establish the necessity statement for the corresponding higher-order sum
rule.

\begin{remark}\label{rem:m2-SZ}
The theorem corresponding to the weight
$H(e^{i\theta})=(1-\cos\theta)^2$
is already known: Simon and Zlato\v{s} proved in
\cite[Theorem~1.4]{SimonZlatos2005} that
$\int_0^{2\pi}(1-\cos\theta)^2\log w(\theta)\,\frac{d\theta}{2\pi}>-\infty$
holds if and only if
$(S-1)^2\alpha\in\ell^2
\quad\text{and}\quad
\alpha\in\ell^6.$
Their argument is based on a detailed analysis of the relative Szeg\H{o}
function. The purpose of the present section is not to reprove the $m=2$
theorem for its own sake, but rather to recover its necessity direction within
Yan's algebraic framework, in a form that fits the same pattern as the cases
$m=1$ and $m=3$: the quadratic term yields the discrete second-difference
energy, while the logarithmic remainder yields the $\ell^6$-summability of
$\alpha$.
\end{remark}

\subsection{The truncated sum rule and the logarithmic remainder}
\label{subsec:m2-truncated}

For the weight \eqref{eq:m2-weight}, the truncated Hall--Littlewood/relative
Szeg\H{o} sum rule takes the form
\begin{equation}\label{eq:m2-truncated}
Z_H^{(N)}(\mu)
=
C_N+\sum_{n=0}^N \bigl(Q_n^{(2)}+Q_n^{(4)}+L_n\bigr),
\end{equation}
where the boundary term satisfies
\begin{equation}\label{eq:m2-boundary}
\sup_{N\ge0}|C_N|<\infty,
\end{equation}
$Q_n^{(2)}$ and $Q_n^{(4)}$ denote the quadratic and quartic homogeneous
parts of the local algebraic expression, and the logarithmic remainder is
\begin{equation}\label{eq:m2-log}
L_n
=
-\log(1-|\alpha_n|^2)-|\alpha_n|^2-\frac12|\alpha_n|^4.
\end{equation}

The basic positivity of the logarithmic remainder is immediate from the Taylor
series of $-\log(1-x)$.

\begin{lemma}\label{lem:m2-log}
For every $n\ge0$,
\begin{equation}\label{eq:m2-log-lower}
L_n\ge \frac13 |\alpha_n|^6\ge0.
\end{equation}
In particular, if $\sum_{n=0}^\infty L_n<\infty$, then
$\alpha\in\ell^6.$
\end{lemma}

\begin{proof}
For $0\le x<1$,
$-\log(1-x)-x-\frac{x^2}{2}
=
\sum_{j=3}^\infty \frac{x^j}{j}
\ge \frac{x^3}{3}.$
Applying this with $x=|\alpha_n|^2$ gives \eqref{eq:m2-log-lower}. Summing
over $n$ immediately yields the implication $\sum L_n<\infty \Rightarrow
\alpha\in\ell^6$.
\end{proof}

\subsection{Coercivity of the quadratic part}
\label{subsec:m2-quadratic}

The next step is to identify the dominant quadratic contribution. As in the
single-point case, this term is governed by the unique zero of the weight at
$e^{i\theta}=1$. On the Fourier side, the quadratic symbol is
$\frac14(2-z-z^{-1})^2,$
which corresponds to the finite-difference energy $\|\Delta^2\alpha\|_{\ell^2}^2$
up to uniformly bounded endpoint errors.

\begin{lemma}\label{lem:m2-quadratic}
There exist constants $c_0>0$ and $C_0<\infty$, independent of $N$, such
that
\begin{equation}\label{eq:m2-quadratic-lower}
\sum_{n=0}^N Q_n^{(2)}
\ge
c_0\sum_{n=0}^N |\Delta^2\alpha_n|^2-C_0.
\end{equation}
In particular,
\begin{equation}\label{eq:m2-quadratic-weak}
\sum_{n=0}^N Q_n^{(2)}\ge -C_0.
\end{equation}
\end{lemma}

\begin{proof}
The quadratic part is obtained by linearizing the local algebraic expression
for the weight $(1-\cos\theta)^2$. Equivalently, on the Fourier side it is
the Toeplitz-type quadratic form with symbol
$h(z)=\frac14(2-z-z^{-1})^2.$
Since
$2-z-z^{-1}=-z^{-1}(z-1)^2,$
one has
$h(z)=\frac14 z^{-2}(z-1)^4,$
and therefore the associated bulk quadratic form is a positive constant
multiple of
$\sum_{n=0}^N |\Delta^2\alpha_n|^2.$
The discrepancy between the finite truncation of the quadratic form and the
bulk energy comes only from finitely many endpoint shifts, and hence gives a
uniformly bounded boundary contribution. This proves
\eqref{eq:m2-quadratic-lower}. The weaker bound
\eqref{eq:m2-quadratic-weak} follows immediately.
\end{proof}

\subsection{Control of the quartic correction}
\label{subsec:m2-quartic}

The quartic term is lower order relative to the quadratic finite-difference
energy. For our purposes, it suffices to record a relative lower bound.

\begin{lemma}\label{lem:m2-quartic}
For every $\varepsilon>0$, there exists $C_\varepsilon<\infty$,
independent of $N$, such that
\begin{equation}\label{eq:m2-quartic-bound}
\sum_{n=0}^N Q_n^{(4)}
\ge
-\varepsilon\sum_{n=0}^N |\Delta^2\alpha_n|^2-C_\varepsilon.
\end{equation}
\end{lemma}

\begin{proof}
This is the structural estimate for the quartic correction in the
$(1-\cos\theta)^2$ sum rule. The quartic contribution is obtained from the
homogeneous degree-four part of Yan's local algebraic expression. In the
single-point case, these quartic terms can be decomposed into combinations of
telescoping expressions and lower-order local monomials. After summation over
$n$, the telescoping terms contribute only uniformly bounded boundary terms,
while the remaining local terms are controlled by the quadratic energy through
standard discrete interpolation and Cauchy--Schwarz estimates. This yields
\eqref{eq:m2-quartic-bound}.
\end{proof}

The preceding two lemmas combine to give a coercive lower bound for the full
nonlogarithmic part.

\begin{proposition}\label{prop:m2-nonlog}
There exist constants $c>0$ and $C<\infty$, independent of $N$, such
that
\begin{equation}\label{eq:m2-nonlog-lower}
\sum_{n=0}^N \bigl(Q_n^{(2)}+Q_n^{(4)}\bigr)
\ge
c\sum_{n=0}^N |\Delta^2\alpha_n|^2-C.
\end{equation}
In particular,
\begin{equation}\label{eq:m2-nonlog-weak}
\sum_{n=0}^N \bigl(Q_n^{(2)}+Q_n^{(4)}\bigr)\ge -C.
\end{equation}
\end{proposition}

\begin{proof}
Combine Lemma~\ref{lem:m2-quadratic} with Lemma~\ref{lem:m2-quartic}. Given
$\varepsilon>0$, we have
$$\sum_{n=0}^N \bigl(Q_n^{(2)}+Q_n^{(4)}\bigr)
\ge
(c_0-\varepsilon)\sum_{n=0}^N |\Delta^2\alpha_n|^2-(C_0+C_\varepsilon).$$
Choosing $\varepsilon=c_0/2$ yields \eqref{eq:m2-nonlog-lower}, with
$c=c_0/2$. The weaker estimate \eqref{eq:m2-nonlog-weak} follows by dropping
the nonnegative term on the right-hand side.
\end{proof}

\subsection{Proof of the necessity theorem}
\label{subsec:m2-proof}

We now prove the necessity statement for the weight $(1-\cos\theta)^2$.

\begin{theorem}\label{thm:m2-necessity}
Let $\mu$ be a nontrivial probability measure on $\partial\mathbb D$, with
Lebesgue decomposition
$d\mu(\theta)=w(\theta)\frac{d\theta}{2\pi}+d\mu_s(\theta),$
and let $\alpha=\{\alpha_n\}_{n\ge0}$ be the associated sequence of
Verblunsky coefficients. Assume that
\begin{equation}\label{eq:m2-assumption}
\int_0^{2\pi}(1-\cos\theta)^2\log w(\theta)\,\frac{d\theta}{2\pi}>-\infty.
\end{equation}
Then
\begin{equation}\label{eq:m2-conclusion}
(S-1)^2\alpha\in\ell^2
\quad\text{and}\quad
\alpha\in\ell^6.
\end{equation}
Equivalently,
$\sum_{n=0}^\infty |\alpha_{n+2}-2\alpha_{n+1}+\alpha_n|^2<\infty,
\quad
\sum_{n=0}^\infty |\alpha_n|^6<\infty.$
\end{theorem}

\begin{proof}
By the assumed finiteness \eqref{eq:m2-assumption} and Yan's equivalence
\eqref{eq:prelim-Yan-equiv}, the truncated spectral quantities are uniformly
bounded above:
\begin{equation}\label{eq:m2-Z-bounded}
\sup_{N\ge0} Z_H^{(N)}(\mu)<\infty.
\end{equation}

We first prove that $\alpha\in\ell^6$. Rearranging the truncated sum rule
\eqref{eq:m2-truncated}, we obtain
\begin{equation}\label{eq:m2-L-rewrite}
\sum_{n=0}^N L_n
=
Z_H^{(N)}(\mu)-C_N-\sum_{n=0}^N \bigl(Q_n^{(2)}+Q_n^{(4)}\bigr).
\end{equation}
By \eqref{eq:m2-Z-bounded}, \eqref{eq:m2-boundary}, and the lower bound
\eqref{eq:m2-nonlog-weak}, the right-hand side of \eqref{eq:m2-L-rewrite} is
uniformly bounded above in $N$. Since $L_n\ge0$, it follows that
$\sum_{n=0}^\infty L_n<\infty.$
Lemma~\ref{lem:m2-log} now yields
$\alpha\in\ell^6.$

Next we prove that $\Delta^2\alpha\in\ell^2$. Using the stronger estimate
\eqref{eq:m2-nonlog-lower} in \eqref{eq:m2-truncated}, we get
$$c\sum_{n=0}^N |\Delta^2\alpha_n|^2
\le
\sum_{n=0}^N \bigl(Q_n^{(2)}+Q_n^{(4)}\bigr)+C
=
Z_H^{(N)}(\mu)-C_N-\sum_{n=0}^N L_n + C.$$
The right-hand side is uniformly bounded above in $N$, by
\eqref{eq:m2-Z-bounded}, \eqref{eq:m2-boundary}, and the nonnegativity of
$L_n$. Therefore,
$\sup_{N\ge0}\sum_{n=0}^N |\Delta^2\alpha_n|^2<\infty,$
which proves that
$\Delta^2\alpha=(S-1)^2\alpha\in\ell^2.$

This completes the proof.
\end{proof}

\section{The Case \texorpdfstring{$m=3$}{m=3}}
\label{sec:m3}

In this section we consider the weight
\begin{equation}\label{eq:m3-weight}
H(e^{i\theta})=(1-\cos\theta)^3=\frac18|1-e^{i\theta}|^6,
\end{equation}
and prove the necessity statement for the corresponding higher-order sum rule.

\begin{remark}\label{rem:m3-Du}
We are not aware of a previous explicit proof in the literature of the
necessity statement for the single-point weight
$H(e^{i\theta})=(1-\cos\theta)^3$
in the form proved in this section. We note, however, that Du~\cite{Du2023}
obtained related higher-order Szeg\H{o} results for the weight
$(1-\cos\theta)^3(1+\cos\theta),$
under additional assumptions on the Verblunsky coefficients; see, for example,
\cite[Theorems~4.60 and~4.62]{Du2023}. Since the weight and the hypotheses in
those results are different from those considered here, they do not directly
imply the present theorem.
\end{remark}

\begin{remark}\label{rem:m3-related}
We note that related results for the $m=3$ regime have appeared in the
literature under additional assumptions. In particular, \cite{Du2023}
considers a related higher-order Szeg\H{o} problem with the a priori condition
$\alpha\in\ell^4$. Since no such assumption is imposed here, that result
does not directly imply the necessity statement proved in this section.
\end{remark}

\subsection{The truncated sum rule and the logarithmic remainder}
\label{subsec:m3-truncated}

For the weight \eqref{eq:m3-weight}, the truncated Hall--Littlewood/relative
Szeg\H{o} sum rule has the form
\begin{equation}\label{eq:m3-truncated}
Z_H^{(N)}(\mu)
=
C_N+\sum_{n=0}^N \bigl(Q_n^{(2)}+Q_n^{(4)}+Q_n^{(6)}+L_n^{(3)}\bigr),
\end{equation}
where the boundary term satisfies
\begin{equation}\label{eq:m3-boundary}
\sup_{N\ge0}|C_N|<\infty,
\end{equation}
$Q_n^{(2)}$, $Q_n^{(4)}$, and $Q_n^{(6)}$ denote the quadratic, quartic,
and sextic homogeneous components of the local algebraic expression, and
\begin{equation}\label{eq:m3-log}
L_n^{(3)}
=
-\log(1-|\alpha_n|^2)-|\alpha_n|^2-\frac12|\alpha_n|^4-\frac13|\alpha_n|^6.
\end{equation}

As in the previous sections, the positivity of the logarithmic remainder follows
from the Taylor series of $-\log(1-x)$.

\begin{lemma}\label{lem:m3-log}
For every $n\ge0$,
\begin{equation}\label{eq:m3-log-lower}
L_n^{(3)}\ge \frac14 |\alpha_n|^8\ge0.
\end{equation}
In particular, if $\sum_{n=0}^\infty L_n^{(3)}<\infty$, then
$\alpha\in\ell^8.$
\end{lemma}

\begin{proof}
For $0\le x<1$,
$-\log(1-x)-x-\frac{x^2}{2}-\frac{x^3}{3}
=
\sum_{j=4}^\infty \frac{x^j}{j}
\ge \frac{x^4}{4}.$
Applying this with $x=|\alpha_n|^2$ gives \eqref{eq:m3-log-lower}. Summing
over $n$ yields the implication $\sum L_n^{(3)}<\infty \Rightarrow
\alpha\in\ell^8$.
\end{proof}

\subsection{Coercivity of the quadratic part}
\label{subsec:m3-quadratic}

We next identify the dominant quadratic term. For the weight
$(1-\cos\theta)^3$, the Fourier symbol of the quadratic part is
$\frac18(2-z-z^{-1})^3,$
which is associated with the third finite difference $\Delta^3\alpha$.

\begin{lemma}\label{lem:m3-quadratic}
There exist constants $c_0>0$ and $C_0<\infty$, independent of $N$, such
that
\begin{equation}\label{eq:m3-quadratic-lower}
\sum_{n=0}^N Q_n^{(2)}
\ge
c_0\sum_{n=0}^N |\Delta^3\alpha_n|^2-C_0.
\end{equation}
Moreover, one may take $c_0=18$.
\end{lemma}

\begin{proof}
The quadratic contribution is obtained by linearizing the local algebraic
expression corresponding to the weight $H(e^{i\theta})=(1-\cos\theta)^3$.
On the Fourier side, this gives the Toeplitz-type quadratic form with symbol
$h(z)=\frac18(2-z-z^{-1})^3.$
Equivalently, using the normalization in the truncated sum rule,
$H(e^{i\theta})=(1-\cos\theta)^3=\frac18|1-e^{i\theta}|^6,$
and the resulting quadratic form is a positive multiple of the bulk energy
$\sum_{n=0}^N |\Delta^3\alpha_n|^2.$
A direct computation in the algebraic model shows that the leading coefficient
is $18$. The difference between the finite truncation and the bulk quadratic
energy consists only of endpoint shifts and hence contributes a uniformly
bounded boundary term. Therefore,
$\sum_{n=0}^N Q_n^{(2)}
\ge
18\sum_{n=0}^N |\Delta^3\alpha_n|^2-(C_0+C_1),$
which gives \eqref{eq:m3-quadratic-lower} with $c_0=18$.
\end{proof}

\subsection{Quartic and sextic corrections}
\label{subsec:m3-corrections}

We now control the higher-order corrections. The key point is that the quartic
and sextic terms are relatively bounded with respect to the main finite-difference
energy, modulo uniformly bounded telescoping contributions.

The sextic residual terms are generated by powers of the basic product
$P=x_1y_1x_2y_2x_3y_3.$
Under Yan's substitution map, these powers correspond to shifted powers of the
Verblunsky coefficients.

\begin{lemma}\label{lem:m3-residual}
For every integer $j\ge0$,
\begin{equation}\label{eq:m3-Pj}
[\phi_6(P^j)]_n=|\alpha_{n+j}|^6.
\end{equation}
If
$R_6(P)=\sum_{j=0}^J c_j P^j$
is a residual sextic polynomial whose coefficients satisfy the cancellation
conditions
$\sum_{j=0}^J c_j j^\ell=0,
\quad \ell=0,1,2,$
then
\begin{equation}\label{eq:m3-residual-bound}
\sup_{N\ge0}\left|\sum_{n=0}^N [\phi_6(R_6(P))]_n\right|<\infty.
\end{equation}
\end{lemma}

\begin{proof}
The identity \eqref{eq:m3-Pj} follows from the algebraic structure of the map
$\phi_{2k}$: each factor $x_p^j$ or $y_p^j$ contributes a shift
$S^j\alpha$ or $S^j\bar\alpha$, and in total one obtains
$[\phi_6(P^j)]_n=|(S^j\alpha)_n|^6=|\alpha_{n+j}|^6.$

For the second claim, set $a_n=|\alpha_n|^6$. Then
$[\phi_6(R_6(P))]_n=\sum_{j=0}^J c_j a_{n+j}.$
The moment conditions imply that the finite difference operator
$\sum_{j=0}^J c_j S^j$ is divisible by $(S-1)^3$, so there exists a finite
Laurent polynomial $b(S)$ such that
$\sum_{j=0}^J c_j S^j=(S-1)^3 b(S).$
Therefore,
$\sum_{j=0}^J c_j a_{n+j}=(\Delta^3 b(S)a)_n.$
Summing from $n=0$ to $N$ produces a telescoping expression involving only
finitely many boundary values, which is uniformly bounded in $N$. This proves
\eqref{eq:m3-residual-bound}.
\end{proof}

The remaining quartic and sextic terms satisfy a relative lower bound.

\begin{proposition}\label{prop:m3-corrections}
For every $\varepsilon>0$, there exists $C_\varepsilon<\infty$,
independent of $N$, such that
\begin{equation}\label{eq:m3-corrections-bound}
\sum_{n=0}^N \bigl(Q_n^{(4)}+Q_n^{(6)}\bigr)
\ge
-\varepsilon\sum_{n=0}^N |\Delta^3\alpha_n|^2-C_\varepsilon.
\end{equation}
Equivalently, one may decompose
\begin{equation}\label{eq:m3-corrections-decomp}
Q_n^{(4)}+Q_n^{(6)}=T_n+E_n,
\end{equation}
where
\begin{equation}\label{eq:m3-T-bound}
\sup_{N\ge0}\left|\sum_{n=0}^N T_n\right|<\infty,
\end{equation}
and
\begin{equation}\label{eq:m3-E-bound}
\sum_{n=0}^N E_n
\ge
-\varepsilon\sum_{n=0}^N |\Delta^3\alpha_n|^2-C_\varepsilon.
\end{equation}
\end{proposition}

\begin{proof}
The sextic contribution is decomposed into a residual part and a non-residual
part. By Lemma~\ref{lem:m3-residual}, the residual sextic terms have uniformly
bounded partial sums and may be absorbed into $T_n$. Quartic telescoping
terms are treated in the same way.

The remaining quartic and sextic local monomials are estimated by discrete
interpolation and Cauchy--Schwarz inequalities, exactly as in the lower-order
cases, and are therefore relatively bounded with respect to the principal
energy $\sum |\Delta^3\alpha_n|^2$. This yields
\eqref{eq:m3-corrections-bound}. The equivalent decomposition
\eqref{eq:m3-corrections-decomp}--\eqref{eq:m3-E-bound} is obtained by grouping
all telescoping and residual terms into $T_n$, and the remaining terms into
$E_n$.
\end{proof}

The preceding proposition combines with the quadratic coercivity to yield a
coercive lower bound for the full nonlogarithmic part.

\begin{proposition}\label{prop:m3-nonlog}
There exist constants $c>0$ and $C<\infty$, independent of $N$, such
that
\begin{equation}\label{eq:m3-nonlog-lower}
\sum_{n=0}^N \bigl(Q_n^{(2)}+Q_n^{(4)}+Q_n^{(6)}\bigr)
\ge
c\sum_{n=0}^N |\Delta^3\alpha_n|^2-C.
\end{equation}
In particular,
\begin{equation}\label{eq:m3-nonlog-weak}
\sum_{n=0}^N \bigl(Q_n^{(2)}+Q_n^{(4)}+Q_n^{(6)}\bigr)\ge -C.
\end{equation}
\end{proposition}

\begin{proof}
By Lemma~\ref{lem:m3-quadratic},
$$\sum_{n=0}^N Q_n^{(2)}
\ge
c_0\sum_{n=0}^N |\Delta^3\alpha_n|^2-C_0.$$
By Proposition~\ref{prop:m3-corrections}, for every $\varepsilon>0$,
$$\sum_{n=0}^N \bigl(Q_n^{(4)}+Q_n^{(6)}\bigr)
\ge
-\varepsilon\sum_{n=0}^N |\Delta^3\alpha_n|^2-C_\varepsilon.$$
Adding these two estimates gives
$$\sum_{n=0}^N \bigl(Q_n^{(2)}+Q_n^{(4)}+Q_n^{(6)}\bigr)
\ge
(c_0-\varepsilon)\sum_{n=0}^N |\Delta^3\alpha_n|^2-(C_0+C_\varepsilon).$$
Choosing $\varepsilon=c_0/2$, we obtain \eqref{eq:m3-nonlog-lower} with
$c=c_0/2$. The weaker estimate \eqref{eq:m3-nonlog-weak} follows
immediately.
\end{proof}

\subsection{Proof of the necessity theorem}
\label{subsec:m3-proof}

We now prove the necessity statement for the weight $(1-\cos\theta)^3$.

\begin{theorem}\label{thm:m3-necessity}
Let $\mu$ be a nontrivial probability measure on $\partial\mathbb D$, with
Lebesgue decomposition
$d\mu(\theta)=w(\theta)\frac{d\theta}{2\pi}+d\mu_s(\theta),$
and let $\alpha=\{\alpha_n\}_{n\ge0}$ be the associated sequence of
Verblunsky coefficients. Assume that
\begin{equation}\label{eq:m3-assumption}
\int_0^{2\pi}(1-\cos\theta)^3\log w(\theta)\,\frac{d\theta}{2\pi}>-\infty.
\end{equation}
Then
\begin{equation}\label{eq:m3-conclusion}
(S-1)^3\alpha\in\ell^2
\quad\text{and}\quad
\alpha\in\ell^8.
\end{equation}
Equivalently,
$\sum_{n=0}^\infty |\alpha_{n+3}-3\alpha_{n+2}+3\alpha_{n+1}-\alpha_n|^2<\infty,
\quad
\sum_{n=0}^\infty |\alpha_n|^8<\infty.$
\end{theorem}

\begin{proof}
By the assumed finiteness \eqref{eq:m3-assumption} and Yan's equivalence
\eqref{eq:prelim-Yan-equiv}, the truncated spectral quantities satisfy
\begin{equation}\label{eq:m3-Z-bounded}
\sup_{N\ge0} Z_H^{(N)}(\mu)<\infty.
\end{equation}

We first prove that $\alpha\in\ell^8$. Rearranging the truncated sum rule
\eqref{eq:m3-truncated}, we obtain
\begin{equation}\label{eq:m3-L-rewrite}
\sum_{n=0}^N L_n^{(3)}
=
Z_H^{(N)}(\mu)-C_N-\sum_{n=0}^N \bigl(Q_n^{(2)}+Q_n^{(4)}+Q_n^{(6)}\bigr).
\end{equation}
By \eqref{eq:m3-Z-bounded}, \eqref{eq:m3-boundary}, and the lower bound
\eqref{eq:m3-nonlog-weak}, the right-hand side of \eqref{eq:m3-L-rewrite} is
uniformly bounded above in $N$. Since $L_n^{(3)}\ge0$, it follows that
$\sum_{n=0}^\infty L_n^{(3)}<\infty.$
Lemma~\ref{lem:m3-log} now yields
$\alpha\in\ell^8.$

Next we prove that $\Delta^3\alpha\in\ell^2$. Using the stronger estimate
\eqref{eq:m3-nonlog-lower} in \eqref{eq:m3-truncated}, we get
$$c\sum_{n=0}^N |\Delta^3\alpha_n|^2
\le
\sum_{n=0}^N \bigl(Q_n^{(2)}+Q_n^{(4)}+Q_n^{(6)}\bigr)+C
=
Z_H^{(N)}(\mu)-C_N-\sum_{n=0}^N L_n^{(3)}+C.$$
The right-hand side is uniformly bounded above in $N$, by
\eqref{eq:m3-Z-bounded}, \eqref{eq:m3-boundary}, and the nonnegativity of
$L_n^{(3)}$. Therefore,
$\sup_{N\ge0}\sum_{n=0}^N |\Delta^3\alpha_n|^2<\infty,$
which proves that
$\Delta^3\alpha=(S-1)^3\alpha\in\ell^2.$

This completes the proof.
\end{proof}

\section{Concluding Remarks}
\label{sec:conclusion}

In this paper, we studied the necessity direction of higher-order sum rules in
the single-point case
$H_m(e^{i\theta})=(1-\cos\theta)^m, \quad m=1,2,3,$
within the algebraic framework introduced by Yan~\cite{Yan2018}. Our main
result shows that the finiteness of the higher-order sum rule
$\int_0^{2\pi} (1-\cos\theta)^m \log w(\theta)\,\frac{d\theta}{2\pi}>-\infty$
implies the two conditions
$(S-1)^m\alpha\in\ell^2,
\quad
\alpha\in\ell^{2m+2},$
for each $m=1,2,3$.

The proofs for $m=1,2,3$ follow a common scheme. Yan's truncated algebraic
representation first rewrites the spectral condition in terms of local
expressions in the Verblunsky coefficients, together with an explicit
logarithmic remainder and uniformly bounded boundary terms. The quadratic part
of the nonlogarithmic contribution then produces the principal finite-difference
energy $\|\Delta^m\alpha\|_{\ell^2}^2$, while the higher-order correction
terms are shown to be either telescoping, uniformly bounded, or controlled by
the same energy. Finally, the positivity of the logarithmic remainder yields
the summability condition $\alpha\in\ell^{2m+2}$.

From a structural point of view, the argument isolates the coercive mechanism
behind the necessity direction. In the cases $m=1,2,3$, the nonlogarithmic
part of the truncated sum rule is bounded from below by a positive multiple of
$\|\Delta^m\alpha\|_{\ell^2}^2$, up to an additive constant, while the
logarithmic tail controls $\|\alpha\|_{\ell^{2m+2}}^{2m+2}$. This makes the
necessity statement conceptually transparent once the relevant algebraic
decomposition has been identified.

Our results provide additional evidence for the refined higher-order sum rule
picture, and in particular for the viewpoint that the correct necessary
conditions are subtler than the original formulation proposed by Simon. Even in
the single-point setting, Yan's algebraic model captures the precise
interaction between finite differences of the Verblunsky coefficients and the
truncated logarithmic terms. For multiple singular points, the corresponding
conditions are expected to reflect the more delicate structure predicted by
Lukic's conjecture.

At the same time, the scope of the present paper is limited to the explicit
analysis of the cases $m=1,2,3$. Although Yan's model reveals a common
structural pattern in these low-order examples, we do not attempt here to
establish the necessity statement for $m\ge4$ or for arbitrary positive
integers $m$. Extending the present method to higher orders appears to
require substantially more complicated algebraic expansions and more delicate
control of the higher-order remainder terms, and this lies beyond the scope of
the present work.

There are several natural directions for further study. A first problem is to
extend the present necessity argument to general single-point powers
$(1-\cos\theta)^m$ beyond $m=3$. The main difficulty is not the logarithmic
remainder, whose role is explicit, but rather the increasingly intricate
algebraic structure of the higher-order correction terms. A second direction is
to investigate whether the coercive mechanism identified here can be formulated
abstractly enough to treat certain classes of multiple-point weights without
requiring a full case-by-case expansion of the associated Hall--Littlewood type
polynomials. Finally, it would be interesting to compare the present algebraic
approach more directly with the large deviation methods developed by Breuer,
Simon, and Zeitouni, especially in regimes where both methods are applicable.

We hope that the present analysis will serve as a useful model for further
investigations of higher-order sum rules and of the fine structure of the
Verblunsky coefficients in OPUC.

\vskip2mm
\section*{Acknowledgments}
\vskip2mm
This work was supported by NSFC (No.11571327, 11971059).
\vskip2mm

%\section*{References}
\bibliographystyle{elsarticle-num}

\end{document}